\documentclass[a4paper,12pt]{article}
\usepackage{amsmath,amssymb,amsthm}
\usepackage{amscd}

\usepackage[all]{xy}

\makeatletter
\@addtoreset{equation}{subsection}
\makeatother

\newtheorem{dfn}{Definition}[section]

\newtheorem{thm}[dfn]{Theorem}

\newtheorem{conj}[dfn]{Conjecture}

\newcommand\inv{^{-1}}

\newcommand\Ker{{\rm Ker}}

\newcommand\mt{{\overset\bullet\rightsquigarrow}}

\newcommand\pa{\partial}

\newcommand\qq{{\mathbb{Q}}}

\newcommand\TT{{\widehat{T}}}
\newcommand\TTone{{\widehat{T}_{\geq1}}}
\newcommand\wotimes{{\widehat{\otimes}}}
\newcommand\zz{{\mathbb{Z}}}

\begin{document}

\title{
A tensorial description of the Turaev cobracket on genus 0 compact surfaces
}
\author{Nariya Kawazumi
}
\date{}
\maketitle

\abstract{
We give a tensorial description of the Turaev cobracket 
on any genus 0 compact surface 
by the standard group-like expansion, where 
the Bernoulli numbers appear.
}

\begin{center}
{\bf Introduction}
\end{center}

The free homotopy set of free loops on an oriented surface $S$, 
$\hat\pi = \hat\pi(S) = [S^1, S] = \pi_1(S)/(\text{conjugate})$, 
has rich structures. 
In the classical theory of Riemann surfaces, 
the algebraic intersection number of two free loops 
plays an central role. As a non-commutative generalization 
of the intersection number, 
the Goldman bracket \cite{Go} of two free loops 
appears in the Weil-Petersson symplectic geometry \cite{Wo},
the Poisson structure on the moduli space of flat bundles \cite{Go} and 
the Skein algebra of links in the $3$-manifold $S\times [0,1]$ \cite{T}.
In the case where $S$ is a compact surface with connected boundary,
Kuno and the author \cite{KK1} gave a tensorial description of 
the Goldman bracket, and described Dehn twists on the surface $S$
in terms of the Goldman Lie algebra.
These results are generalized to any compact surfaces 
with non-empty boundary in \cite{MT} \cite{KK2} \cite{KKs}.\par
On the other hand, the Turaev cobracket $\delta$ 
is related to Turaev's earlier work \cite{Tu78}, and was introduced
by Turaev \cite{T} in connection with the Skein algebra. 
It is a dual notion of the Goldman bracket, 
and measures the self-intersection of a single free loop. 
But little is known about the Turaev cobracket. As was discovered 
by Kuno and the author \cite{KK3}, 
the Turaev cobracket gives a geometric constraint of the images
of the (higher) Johnson homomorphisms. In order to 
deduce some results from this fact, we need a tensorial description 
of the Turaev cobracket. In \cite{KK3} and \cite{MT2}, the lowest 
degree term of the description was computed.
When the preprint of this paper \cite{K2} was uploaded at the arXiv 
(June 10, 2015), there was no other full results on the tensorial description. 
\par

In this paper we will give the tensorial description of the Turaev cobracket 
for any genus $0$ compact surface with respect to the standard group-like 
expansion $\theta^{\rm std}$. Unfortunately the expansion $\theta^{\rm std}$
does not reflect the topology of the surface enough, so that we cannot deduce 
topological consequences from our result. \par
The description is stated in Theorem\ref{dstd}, where the Bernoulli numbers 
appear. In this paper, following the convention in \cite{MT}, we agree that 
the function $s(z)$ and the Bernoulli numbers $B_{2m}$ are defined by
$$
\aligned
s(z)  & = \dfrac{1}{e^{-z}-1} + \dfrac{1}{z} 
= -\dfrac12 - \sum^\infty_{m=1}\dfrac{B_{2m}}{(2m)!}z^{2m-1}\\
&= -\frac12 - \frac{1}{12}z + \frac{1}{720}z^3 - \frac{1}{30240}z^5 + \cdots.
\endaligned
$$
The appearance of the Bernoulli numbers comes from 
the tensorial description of the homotopy intersection form by Massuyeau-Turaev
\cite{MT} (Theorem \ref{mt}), and a formula for the coaction operation $\mu$ 
by Fukuhara-Kawazumi-Kuno \cite{FKK} (Theorem \ref{fkk}). 
The Kashiwara-Vergne problem in the formulation by Alekseev-Torossian \cite{AT}
looks for a group-like expansion of the fundamental group of a pair of pants
which is compatible with all the boundary components and satisfies some 
equation involved with the Bernoulli numbers and the divergence cocycle.
As the author announced in \cite{K1}, a regular homotopy version of the 
Turaev cobracket on genus $0$ compact surfaces includes the divergence 
cocycle. Hence the result in this paper seems to suggest the following 
conjecture.

\begin{conj}\label{KV} The tensorial description of the Turaev cobracket 
with respect to any solution to the Kashiwara-Vergne problem is of 
simple expression. In particular, 
the description might be formal, namely, 
might equal its lowest degree term.
\end{conj}
It is our working hypothesis for studying the higher Johnson homomorphisms
that there is a symplectic expansion for a compact surface 
with connected boundary whose description of the Turaev cobracket 
equals the lowest degree term, i.e., Schedler's cobracket \cite{Sch}.
In fact, Kuno \cite{Ku2} 
already found such an expansion for the surface of genus $1$ 
with connected boundary up to degree $10$ by a computer calculation.
If Conjecture \ref{KV} would be true, our hypothesis should be a positive genus 
analogue of the Kashiwara-Vergne problem. \par
\bigskip
After the preprint of this paper was uploaded,
Alekseev, Kuno, Naef and the author \cite{AKKN}
obtained a formal description of the Turaev cobracket 
by regarding solutions of the Kashiwara-Vergne 
problem as special expansions for genus $0$ compact surfaces.
This means that Conjecture \ref{KV} is true.
Independently from our results, 
Massuyeau \cite{M2} obtained a formal description 
of the Turaev cobracket for genus $0$ compact surfaces
by the Kontsevich integral.
\par
Theorem \ref{mt} in this paper is a modification of 
a theorem of Massuyeau and Turaev \cite{MT}. It says that 
the value of a group-like expansion at the boundary loop 
of a surface with connected boundary completely determines the 
tensorial description of the homotopy intersection form
by the expansion. As is showed by Naef \cite{N}, 
this fact can be generalized in the light of a non-commutative 
Poisson geometry, which is one of the foundations of the work
\cite{AKKN}.
\par
\bigskip
The author thanks 
Anton Alekseev,  
Yusuke Kuno, Florian Naef and Shunsuke Tsuji for 
helpful discussions.
The first draft of this paper was written during my stay at IRMA, Strasbourg,
on the occasion of the JSPS-CNRS joint project on Teichm\"uller spaces 
and surface mapping class groups. 
He would like to express his gratitude to IRMA for kind hospitality.  
He is partially supported by the Grant-in-Aid for Scientific Research (S) (No.24224002), (B) (No.24340010) and (B) (No.15H03617) from the Japan Society for Promotion of Sciences. 


\section{Statement of the Result}\label{SR}

Let $S$ be a compact connected oriented surface with non-empty boundary.
It is classified by its genus and the number of its boundary components, 
so that we may denote the surface $S$ by the symbol 
$\Sigma_{g, n+1}$ for some $g, n \geq 0$. 
Here the genus of $S$ is $g$, and the number of the boundary components is $n+1$. 
The fundamental group of the surface $S$ is free of rank $2g+n$. 
In general, for a free group $\pi$ of finite rank, we have the notion of 
group-like expansion. See \cite{M}. In order to recall the definition of a group-like 
expansion, we need to prepare some tensor algebra.
Let $H$ be the first rational homology group of $\pi$, 
i.e., $H := (\pi/[\pi,\pi])\otimes_\zz\qq$.
We denote $[\gamma] := (\gamma\bmod[\pi,\pi])\otimes 1 \in H$ 
for any $\gamma \in \pi$. The completed tensor algebra 
$\TT = \TT(H) := \prod^\infty_{m=0}H^{\otimes m}$ is endowed with the topology
by the decreasing filtration $\TT_{\geq p} := \prod_{m\geq p}H^{\otimes m}$, 
$p \geq 1$, and has the strucuture of a complete Hopf algebra 
with an augmentation $\varepsilon: \TT \to \qq$, 
a coproduct $\Delta: \TT \to \TT\wotimes\TT$ and 
an antipode $\iota: \TT \to \TT$.
They are defined to be the unique continuous algebra (anti)-homomorphisms 
satisfying $\varepsilon(X) = 0$, $\Delta(X) = X\wotimes 1 + 1\wotimes X$
and $\iota(X) = -X$ for any $X \in H$, respectively. The group ring $\qq\pi$ is also 
a Hopf algebra. The augmentation $\varepsilon: \qq\pi \to \qq$, 
the coproduct $\qq\pi \to \qq\pi\otimes\qq\pi $ and 
the antipode $\iota: \qq\pi  \to \qq\pi$ are the 
unique algebra (anti)-homomorphisms 
satisfying $\varepsilon(\gamma) = 1$, $\Delta(\gamma) = \gamma\otimes\gamma$
and $\iota(\gamma) = \gamma\inv$ for any $\gamma \in \pi$, respectively. 
The completion of $\qq\pi$ with respect to the augmentation ideal $I\pi := \Ker \,
\varepsilon$, 
$
\widehat{\qq\pi} := \varprojlim_{p\to\infty}\qq\pi/(I\pi)^p,
$
is a complete Hopf algebra in a natural way. 
\begin{dfn}[See \cite{MT}] The map $\theta: \pi \to \TT$ is a group-like expansion
if the following three conditions hold:
\begin{enumerate}
\item The map $\theta$ is multiplicative, i.e., we have $\theta(\gamma_1\gamma_2) 
= \theta(\gamma_1)\theta(\gamma_2)$ 
for any $\gamma_1$ and $\gamma_2 \in \pi$.
\item For any $\gamma \in \pi$, $\theta(\gamma) \equiv 1 + [\gamma] 
\pmod{\TT_{\geq 2}}$.
\item For any $\gamma \in \pi$, $\theta(\gamma) \in \TT$ is group-like, i.e., 
$\Delta\theta(\gamma) =\theta(\gamma)\wotimes\theta(\gamma)
\in \TT\wotimes\TT$.
\end{enumerate}
\end{dfn}
The linear extension of any group-like expansion induces an isomorphism 
of complete Hopf algebras $\theta: \widehat{\qq\pi}\overset\cong\longrightarrow 
\TT$, $\sum a_\gamma \gamma \mapsto \sum a_\gamma \theta(\gamma)$. 
\par
The group-like expansion we study in this paper is 
defined as follows. Let $S$ be the genus $0$ compact surface $\Sigma_{0, n+1}$ for some $n \geq 0$. Number the boundary components 
as $\pa S = \coprod^n_{k=0}\pa_kS$, and choose a basepoint $* \in \pa_0S$.
The standard generators $\gamma_k \in \pi_1(S, *)$, $1 \leq k \leq n$, 
are given such that each $\gamma_k$ is a simple loop going 
around the $k$-th boundary $\pa_kS$ in the positive direction, 
and the product $\gamma_1\gamma_2\cdots
\gamma_n \in \pi_1(S, *)$ is homotopic to a simple loop 
around the $0$-th boundary $\pa_0S$ in the negative direction. 
Here we read the product $\gamma_1\gamma_2\cdots
\gamma_n$ as a loop going along first $\gamma_1$, next $\gamma_2$,
and finally $\gamma_n$. 
Here we remark that $\epsilon(\overset\cdot\gamma_k(0), 
\overset\cdot\gamma_k(1)) = +1$. 
The fundamental group $\pi_1(S, *)$ is 
a free group of rank $n$ with free generators $\gamma_k$, $1 \leq k \leq n$. 
We denote by $x_k := [\gamma_k] \in H = 
H_1(S; \qq)$, $1\leq k \leq n$, the homology class of $\gamma_k$.
Equivalently $x_k$ is the homology class of the $k$-th boundary $\pa_kS$, 
so that we define $x_0 := [\pa_0S] = - [\gamma_1\gamma_2\cdots
\gamma_n] = - \sum^n_{k=1}x_k \in H = H_1(S; \qq)$. 
Then we can consider the exponential $e^{x_k} = \exp(x_k) 
= \sum^\infty_{m=0} \frac{1}{m!}{x_k}^m \in \TT = \TT(H_1(S; \qq))$. 
We define {\it the standard group-like expansion} $\theta^{\rm std}: \pi = \pi_1(S, *) 
\to \TT = \TT(H_1(S; \qq))$ as the unique group-expansion 
satisfying $\theta^{\rm std}(\gamma_k) = e^{x_k}$, $1 \leq \forall k \leq n$.
Here we require these conditions only for $k \geq 1$, not for $k=0$.
The reason why one can compute the tensorial description of the Turaev cobracket
with respect to the expansion $\theta^{\rm std}$ is that 
we can apply Theorem \ref{fkk} to $x_k = \theta^{\rm std}(\log(\gamma_k))$. \par
Let $\delta: \zz\hat\pi' \to \zz\hat\pi' \otimes \zz\hat\pi'$ be the 
Turaev cobracket \cite{T}. Here $\zz\hat\pi' := \zz\hat\pi/\zz\mathbf{1}$ 
is the quotient of the $\zz$-free module over the set $\hat\pi$, $\zz\hat\pi$, 
by the linear span of the constant loop $\mathbf{1} \in \hat\pi$. 
We denote by $\vert\,\,\vert': \zz\pi_1(S, p) \to 
\zz\hat\pi \to \zz\pi/\zz\mathbf{1}
= \zz\hat\pi'$ the quotient map for any $p \in S$. 
The definition of the Turaev cobracket 
will be stated in \S\ref{prel}. The Goldman bracket and the Turaev cobracket
make $\zz\hat\pi'$ a Lie bialgebra in the sense of Drinfel'd \cite{T}, 
so that we call it the Goldman-Turaev Lie bialgebra of the surface $S$. 
The bialgebra has a completion with respect to the augmentation ideal
$I \pi$, $\widehat{\qq\hat\pi} := \varprojlim_{p\to\infty}\qq\hat\pi'/
\vert (I\pi)^p\vert'$. 
We have a natural continuous extension $\vert\,\,\vert': \widehat{\qq
\pi} \to \widehat{\qq\hat\pi}$. 
The Goldman bracket and the Turaev cobracket 
extend continuously to $\widehat{\qq\hat\pi}$ \cite{KK2}\cite{KK3}. 
In particular, the Turaev cobracket is a continuous map $\delta: \widehat{\qq\hat\pi}
\to \widehat{\qq\hat\pi}\wotimes \widehat{\qq\hat\pi}$. \par
On the tensor algebra side, we denote by $N(\TT)$ the quotient of $\TT$
by the closure of $\qq 1 + [\TT, \TT]$, where $[\TT, \TT]$ is the  
$\qq$-linear subspace of $\TT$ generated by the set 
$\{uv-vu; \,\, u, v \in \TT\}$. The vector space $N(\TT)$ is naturally 
isomorphic to the space of cyclic invariants $\prod^\infty_{m=1}
(H^{\otimes m})^{\zz/m}$, where the cyclic group $\zz/m$ acts on 
the space $H^{\otimes m}$ by cyclic permutation. 
We denote by $\vert\,\,\vert': \TT \to N(\TT)$ the quotient map. 
Any group-like expansion $\theta$ induces a topological isomorphism
$\theta: \widehat{\qq\hat\pi}\overset\cong\rightarrow N(\TT)$ \cite{KK2}.
Thus we have the tensorial description $\delta^\theta$ of the Turaev cobracket
with respect to $\theta$ defined by the diagram
$$\begin{CD}
\widehat{\qq\hat\pi} @>{\delta}>> \widehat{\qq\hat\pi}\wotimes \widehat{\qq\hat\pi}\\
@V{\theta}VV @V{\theta\wotimes\theta}VV\\
N(\TT) @>{\delta^\theta}>> N(\TT)\wotimes N(\TT).
\end{CD}
$$
Now we can formulate our result.
\begin{thm}\label{dstd} Let $\delta^{\rm std} = \delta^{\theta^{\rm std}}$ 
be the tensorial description of the Turaev cobracket with respect to 
the standard group-like expansion $\theta^{\rm std}$ for the surface 
$S = \Sigma_{0, n+1}$. Then, for any $m \geq 1$ and any
$k_1, k_2, \dots, k_m \in \{1, 2, \dots, n\}$,  we have
$$
\aligned
\delta^{\rm std}(x_{k_1}x_{k_2}\cdots& x_{k_m})\\
=  {\rm alt}(\vert\,\,\vert'\otimes\vert\,\,\vert')\Big(&
\sum_{1\leq i < j \leq m}K_{k_ik_j}
(x_{k_{j+1}}\cdots x_{k_m}x_{k_1}\cdots x_{k_{i-1}}
\wotimes x_{k_{i+1}}\cdots x_{k_{j-1}})
\\
& -\frac12\sum^m_{i=1}x_{k_1}\cdots x_{k_{i-1}}x_{k_{i+1}}\cdots x_{k_m}
\wotimes {x_{k_i}}\\
& + \sum^m_{i=1}\sum^\infty_{q=1}\frac{B_{2q}}{(2q)!}\sum^{2q-1}_{p=0}
(-1)^p\begin{pmatrix} 2q\\ p\end{pmatrix}
x_{k_1}\cdots x_{k_{i-1}}{x_{k_i}}^px_{k_{i+1}}\cdots x_{k_m}\wotimes
{x_{k_i}}^{2q-p}\Big).
\endaligned
$$
Here, for $1\leq k, l\leq n$, we denote
$$
K_{k,l} := (1\wotimes \iota)\Delta\left(\epsilon_{kl}x_{k}x_{l} - 
\delta_{kl}\frac{{x_{k}}^2}{e^{-x_{k}}-1}\right) \in \TT\wotimes\TT,
$$
where $\delta_{kl}$ is the Kronecker delta, $\epsilon_{kl}$ is defined by
$$
\epsilon_{kl} := \begin{cases}
1, & \text{if $k > l$,}\\
0, & \text{if $k \leq l$,}
\end{cases}
$$
and ${\rm alt}: N(\TT)\wotimes N(\TT) \to N(\TT)\wotimes N(\TT)$, 
$u\wotimes v \mapsto u\wotimes v-v\wotimes u$, is the alternating operator. 
\end{thm}

\section{Preliminaries}\label{prel}

Let $S$ be a compact connected oriented surface with non-empty boundary.
Choose a basepoint $* \in \pa S$, and denote $\pi := \pi_1(S, *)$.
We begin by recalling the coaction $\mu: \zz\pi \to \zz\pi \otimes
\zz\hat\pi'$, which is introduced in \cite{KK3} inspired by a construction
of Turaev \cite{Tu78}. The alternating part of $\mu$ is just the Turaev 
cobracket $\delta$, but $\mu$ is of multiplicative nature as stated below.
Choose another point $*^+ \in \pa S$ near $*$ in the positive direction. 
For any $\gamma \in \pi$ we regard it as a path from $*$ to $*^+$, and 
choose a representative of $\gamma$ in general position. 
By abuse of notation, we also denote by $\gamma$ 
the representative. Then the curve $\gamma$ is an immersion, and 
its singularities are at worst transverse double points. 
For each double point $p$ of $\gamma$ we have a unique pair 
$0 < t^p_1 < t^p_2 < 1$ of parameters 
such that $\gamma(t^p_1) = \gamma(t^p_2) 
= p$. Then $\mu(\gamma) \in \zz\pi\otimes\zz\pi'$ is defined by
$$
\mu(\gamma) := -\sum_p \varepsilon(\overset\cdot\gamma(t^p_1), 
\overset\cdot\gamma(t^p_2))(\gamma_{0t^p_1}\gamma_{t^p_21})
\otimes \vert \gamma_{t^p_1t^p_2}\vert',
$$
where the sum runs over the set of self-intersection points of $\gamma$, 
$\varepsilon(\overset\cdot\gamma(t^p_1), 
\overset\cdot\gamma(t^p_2)) \in \{\pm1\}$ is the local intersection number
with respect to the orientation of $S$, and $\gamma_{s_1s_2}$ is the 
restriction of $\gamma$ to the interval $[s_1, s_2]\subset [0,1]$
for any $0 \leq s_1 < s_2 \leq 1$.  
The operation $\mu$ is well-defined, i.e., independent of the choice of 
a representative \cite{KK3}. The Turaev cobracket 
$\delta: \zz\hat\pi' \to \zz\hat\pi'\otimes \zz\hat\pi'$ \cite{T}
can be defined to be the alternating part of $\mu$
\begin{equation}
\delta\circ \vert\,\,\vert' = {\rm alt}\circ (1\otimes \vert\,\,\vert')\circ\mu:
\zz\pi \to \zz\hat\pi'\otimes \zz\hat\pi'.
\label{20}
\end{equation}
Here ${\rm alt}: \zz\hat\pi'\otimes \zz\hat\pi' \to \zz\hat\pi'\otimes \zz\hat\pi'$
is the alternating operator as above. The map $\mu$ extends continuously to 
the map
$\mu: \widehat{\qq\pi} \to \widehat{\qq\pi}\wotimes \widehat{\qq\hat\pi}.$
For example, the extension $\mu$ is computed as follows.
\begin{thm}[\cite{FKK}]
\label{fkk}
If $\gamma \in \pi_1(S, *)$ is represented by a simple loop 
with $\varepsilon(\overset\cdot{\gamma}(0), 
\overset\cdot{\gamma}(1)) = +1$, then we have 
$$
\mu(\log \gamma) = \frac{1}{2}1\wotimes\vert\log \gamma\vert' 
+ \sum^\infty_{m=1}\frac{B_{2m}}{(2m)!}\sum^{2m-1}_{p=0}\begin{pmatrix}
2m\\ p\end{pmatrix}(-1)^p(\log \gamma)^p\wotimes 
\vert(\log \gamma)^{2m-p}\vert'.
$$
\end{thm}

We can define the tensorial description of the map $\mu^\theta: 
\TT \to \TT\wotimes N(\TT)$ with respect to any group-like expansion 
$\theta$ of the fundamental group $\pi_1(S, *)$. 
Theorem \ref{dstd} follows immediately from the following.
\begin{thm}\label{mustd}
Let $\delta^{\rm std} = \delta^{\theta^{\rm std}}$ 
be the tensorial description of the Turaev cobracket with respect to 
the standard group-like expansion $\theta^{\rm std}$ for the surface 
$S = \Sigma_{0, n+1}$. Then, for any $m \geq 1$ and any
$k_1, k_2, \dots, k_m \in \{1, 2, \dots, n\}$,  we have
$$
\aligned
\mu^{\rm std}(x_{k_1}x_{k_2}&\cdots x_{k_m})\\
=  (1\otimes\vert\,\,\vert')\Big(&
\sum_{1\leq i < j \leq m}(
x_{k_1}\cdots x_{k_{i-1}}\wotimes 1)K_{k_ik_j}
(x_{k_{j+1}}\cdots x_{k_m}\wotimes x_{k_{i+1}}\cdots x_{k_{j-1}})
\\
& -\frac12\sum^m_{i=1}x_{k_1}\cdots x_{k_{i-1}}x_{k_{i+1}}\cdots x_{k_m}
\wotimes {x_{k_i}}\\
& + \sum^m_{i=1}\sum^\infty_{q=1}\frac{B_{2q}}{(2q)!}\sum^{2q-1}_{p=0}
(-1)^p\begin{pmatrix} 2q\\ p\end{pmatrix}
x_{k_1}\cdots x_{k_{i-1}}{x_{k_i}}^px_{k_{i+1}}\cdots x_{k_m}\wotimes
{x_{k_i}}^{2q-p}\Big).
\endaligned
$$
\end{thm}
Here it should be
remarked $\vert x_{k_1}\cdots x_{k_{i-1}}x_{k_{j+1}}\cdots x_{k_m}\vert' 
= \vert x_{k_{j+1}}\cdots x_{k_m}x_{k_1}\cdots x_{k_{i-1}}\vert' \in N(\TT)$. 
The rest of this paper is devoted to the proof of Theorem \ref{mustd}.\par
Our proof consists of Theorem \ref{Y}, Theorem \ref{fkk}, and
(a slight modification of) the tensorial 
description of the homotopy intersection form by Massuyeau-Turaev \cite{MT}, 
which we will explain later in short. 
Let $S$ be a (general) connected compact oriented surface with non-empty boundary.
Choose basepoints $*$ and $*^+$ in $\pa S$ as above.
Then, using a short path along the boundary from $*$ to $*^+$, 
we identify the fundamental groups $\pi = \pi_1(S, *)$ and $\pi_1(S, *^+)$
with the homotopy set of path from $*$ to $*^+$ and that from $*^+$ to $*$.
Then the homotopy intersection form 
$\eta: \zz\pi_1(S, *) \otimes
\zz\pi_1(S, *^+) \to \zz\pi$, introduced by Papakyriakopoulos 
\cite{Pa} and Turaev \cite{Tu78} independently, is defined as follows.
For $\gamma_1 \in \pi_1(S, *)$ and $\gamma_2 \in \pi_1(S, *^+)$ 
we choose their representatives in general position.
Then $\eta(\gamma_1, \gamma_2) \in \zz\pi$ is defined by 
$$
\eta(\gamma_1, \gamma_2) := \sum_{p \in \gamma\cap\delta}
\varepsilon_p(\gamma_1, \gamma_2)(\gamma_1)_{*p}(\gamma_2)_{p*^+},
$$
where $\varepsilon_p(\gamma_1, \gamma_2) \in \{\pm1\}$ is 
the local intersection number of $\gamma_1$ and $\gamma_2$ 
at the intersection point $p$, $(\gamma_1)_{*p}$ the segment
of $\gamma_1$ from $*$ to $p$, and $(\gamma_2)_{p*^+}$ 
that of $\gamma_2$ from $p$ to $*^+$. We define a map
$\kappa: \zz\pi\otimes \zz\pi \to \zz\pi\otimes \zz\pi$ 
by 
$$
\kappa(\gamma_1, \gamma_2) := 
- (1\otimes\gamma_2)\left(
(1\otimes\iota)\Delta\eta(\gamma_1, \gamma_2)
\right)(1\otimes\gamma_1)
$$
for $\gamma_1, \gamma_2 \in \pi$. 
In other words, if we denote $\Delta u = \sum u'\otimes u''$ and 
$\Delta v = \sum v'\otimes v''$ for $u, v \in \qq\pi$, we define
\begin{eqnarray}
\kappa(u,v) = -\sum(1\otimes v'')\left(
(1\otimes\iota)\Delta\eta(u', v')
\right)(1\otimes u'').
\label{kp}
\end{eqnarray}
Then we have a product formula
$$
\mu(\gamma_1\gamma_2) = \mu(\gamma_1)(\gamma_2\otimes 1)
+ (\gamma_1\otimes 1)\mu(\gamma_2) 
+ (1\otimes\vert\,\,\vert')\kappa(\gamma_1, \gamma_2).
$$
More generally, we have 
$$
\aligned
&\mu(u_1u_2\cdots u_m)\\
=& \sum^m_{i=1}((u_1\cdots u_{i-1})\otimes 1) \mu(u_i)
((u_{i+1}\cdots u_m)\otimes 1)\\
&+ \sum_{i<j}
((u_1\cdots u_{i-1})\otimes 1) (1\otimes\vert\,\,\vert')
\left(\kappa(u_i, u_j)
(u_{j+1}\cdots u_m\otimes u_{i+1}\cdots u_{j-1})\right)
\endaligned
\eqno{(*)}
$$
for any $m \geq 1$ and any $u_1, u_2, \dots, u_m \in \zz\pi$
\cite{KK3} (Corollary 4.3.4).\par
Massuyeau  and Turaev \cite{MT} gave explicitly the tensorial description
of the homotopy intersection form $\eta$ with respect to 
any symplectic expansion \cite{M} in the case $S = \Sigma_{g,1}$, 
$g \geq 1$, i.e., the boundary $\pa S$ is connected.
In this case, we denote by $\star \in \pa S$ a basepoint on the boundary, 
and by $\zeta \in \pi_1(S, \star)$ the simple loop 
along the boundary in the negative orientation.
The algebraic intersection number $H\otimes H \to \qq$, $X\otimes Y 
\mapsto X\cdot Y$, is a non-degenerate pairing on $H$. The symplectic form
$\omega := \sum^g_{i=1}A_iB_i-B_iA_i \in H^{\otimes 2} \subset \TT$ 
is independent of the choice of a symplectic basis $\{A_i, B_i\}^g_{i=1}
\subset H = H_1(\Sigma_{g,1}; \qq)$. 
Throughout this paper we omit the symbol $\otimes$ 
when it indicates the product in $\TT$.
We have $\theta(\zeta) \equiv 
1 + \omega \pmod{\TT_{\geq 3}}$ for any group-like expansion $\theta$. 
Massuyeau \cite{M} introduced 
the notion of a symplectic expansion: 
A group-like expansion $\theta: \pi \to \TT$ is {\it symplectic} if 
$\theta(\zeta) = \exp(\omega) (= \sum^\infty_{m=0}\frac{1}{m!}
\omega^m) \in \TT$, i.e., $\log\theta(\zeta) = \omega \in \TT$. 
Symplectic expansions (in rational coefficients) exist \cite{M} \cite{Ku1}.
See also \cite{K0} for symplectic expansions in real coefficients.
While their result deals only with symplectic expansions, but it is not hard to 
generalize it to any group-like expansion. \par
In order to give the tensorial description, 
Massuyeau and Turaev \cite{MT} introduced a continuous operation 
$\mt: \TT_{\geq1}
\times \TT_{\geq1} \to \TT$ by 
$$
(X_1\cdots X_{l-1}X_l)\mt (Y_1Y_2\cdots Y_m) := (X_l\cdot Y_1)X_1\cdots X_{l-1}
Y_2\cdots Y_m
$$
for any $l, m \geq 1$ and any $X_i$, $Y_j \in H = H_1(\Sigma_{g,1}; \qq)$. 
Minus the sympletic form is the unit for the operation $\mt$, i.e., 
$(-\omega)\mt u = u\mt (-\omega) = u$ for any $u \in \TT_{\geq1}$.
The restriction of $\mt$ to $\TT_{\geq2}$ is associative, and $\mt(\TT_{\geq l}
\times \TT_{\geq m}) \subset \TT_{\geq(l+m-2)}$. Hence, for any $Z \in 
(-\omega)+\TT_{\geq 3}$, there exists a unique $Z\inv \in 
(-\omega)+\TT_{\geq 3}$ such that $Z\mt Z\inv = Z\inv\mt Z = -\omega$.

\begin{thm}[Massuyeau-Turaev \cite{MT}]\label{mt}
Let $\theta: \pi_1(\Sigma_{g,1}, \star) \to \TT$ be a group-like expansion.
We denote $\Omega = \Omega^\theta 
:= \log\theta(\zeta) \in \omega + \TT_{\geq3}$.
Then the tensorial description of the homotopy intersection form $\eta: 
\widehat{\mathbb{Q}\pi}\times\widehat{\mathbb{Q}\pi} 
\to \widehat{\mathbb{Q}\pi}$ with respect to the expansion $\theta$, 
$\rho^\theta$, is given by 
$$
\rho^\theta(a, b) = (a- \varepsilon(a))\mt ((-\Omega)\inv + \omega
s(\Omega)\omega)\mt (b - \varepsilon(b))
$$
for any $a, b \in \TT$. 
\end{thm}
\begin{proof}
We modify the proof of Theorem 10.4 in Massuyeau-Turaev \cite{MT}. 
The tensorial description $\rho^\theta$ is characterized by the condition
\begin{equation}
\forall X \in H, \quad \rho^\theta(X, e^{-\Omega}) = X.
\label{21}
\end{equation}
Since $s(z)z -1 = z(e^{-z}-1)\inv$, we have $$
\rho^\theta(X, e^{-\Omega}) 
= \rho^\theta(X, \Omega)\frac{e^{-\Omega}-1}{\Omega}  
= \rho^\theta(X, \Omega)(s(\Omega)\Omega -1)\inv.
$$
Hence the condition (\ref{21}) is equivalent to 
\begin{equation}
\forall X \in H, \quad \rho^\theta(X, \Omega) = Xs(\Omega)\Omega - X.
\label{22}
\end{equation}
Now the map $(a, b) \in \TT\times\TT \mapsto (a-\varepsilon(a))s(\Omega)
(b-\varepsilon(b)) \in \TT$ is a Fox pairing in the sense of Massuyeau-Turaev
\cite{MT}. Hence, if we introduce a unique Fox pairing $\rho_\Omega: \TT\times 
\TT \to \TT$ characterized by the condition 
\begin{equation}
\forall X \in H, \quad \rho_\Omega(X, \Omega) = -X,
\label{23}
\end{equation}
then we have 
$$
\rho^\theta(a, b) = \rho_\Omega(a, b) + (a-\varepsilon(a))s(\Omega)
(b-\varepsilon(b))
$$
for any $a$ and $b \in \TT$. Let $\{A_i, B_i\}^g_{i=1} \subset H$ 
be a symplectic basis. The tensor
$$
\aligned
R_\Omega := \sum^g_{i,j=1}(& -B_i\rho_\Omega(A_i, A_j)B_j
+B_i\rho_\Omega(A_i, B_j)A_j\\
&+A_i\rho_\Omega(B_i, A_j)B_j
-A_i\rho_\Omega(B_i, B_j)A_j) \in \TT_{\geq2}
\endaligned
$$
satisfies $\rho_\Omega(a, b) = (a-\varepsilon(a))\mt R_\Omega\mt
(b-\varepsilon(b))$ for any $a$ and $b \in \TT$. Then the condition 
(\ref{23}) is equivalent to $R_\Omega\mt\Omega = \omega$. 
This means $R_\Omega = (-\Omega)\inv$. Therefore we have
$$
\aligned
\rho^\theta(a, b) &= (a-\varepsilon(a))\mt R_\Omega\mt(b-\varepsilon(b))
+ (a-\varepsilon(a))s(\Omega)(b-\varepsilon(b))\\
&= 
(a-\varepsilon(a))\mt((-\Omega)\inv + \omega s(\Omega)\omega)\mt
(b-\varepsilon(b)).
\endaligned
$$
This proves the theorem.
\end{proof}

\section{Proof of the Result}

Now we begin the proof of Theorem \ref{mustd}, from which 
Theorem \ref{dstd} follows immediately by (\ref{20}). 
Let $S$ be the genus $0$ compact surface $\Sigma_{0, n+1}$ 
for some $n \geq 0$. We consider the standard group-like expansion
$\theta^{\rm std}: \pi = \pi_1(S, *) \to \TT = \TT(H_1(\Sigma_{0, n+1}; \qq))$.
Choose one point $*_k \in \pa_kS$ for each component $\pa_kS$ and 
let $\xi_k \in \pi_1(S, *_k)$ be the simple positive boundary loop 
for $1 \leq k \leq n$. We can choose a simple path $\chi_k$ from $* \in \pa_0S$
to $*_k$ such that $\chi_k\xi_k{\chi_k}\inv = \gamma_k \in \pi_1(S, *)$. 
We glue $n$ copies of the surface $\Sigma_{1,1}$ 
to the surface $S = \Sigma_{0,n+1}$ along the boundary $\pa_kS$, 
$1 \leq k \leq n$, such that the basepoints $\star$ and $*_k$ are identified 
with each other. This gluing yields a surface $\hat S \cong \Sigma_{n,1}$. 
Let $\{\alpha_k, \beta_k\}$ be a symplectic generator 
of the fundamental group of the $k$-th copy of $\Sigma_{1,1}$ with basepoint
$\star$. Then the set $\{\chi_k\alpha_k{\chi_k}\inv, 
\chi_k\beta_k{\chi_k}\inv\}^n_{k=1}$ is a symplectic generator of the fundamental 
group $\pi_1(\hat S, *)$. If we denote $A_k := [\chi_k\alpha_k{\chi_k}\inv]$
and $B_k := [\chi_k\beta_k{\chi_k}\inv] \in H_1(\hat S; \qq)$, then 
the set $\{A_k, B_k\}^g_{k=1}$ is a symplectic basis of the homology group
$H_1(\hat S; \qq)$. The map $\imath: \TT = \TT(H_1(S; \qq)) \to 
\TT(H_1(\hat S; \qq))$ defined by $\imath(x_k) := A_kB_k - B_kA_k$ 
is an {\it injective} algebra homomorphism. See \cite{KK2} \S6.2. \par
Let $\theta_k: \pi_1(\Sigma_{1,1}, \star) \to \TT(H_1(\Sigma_{1,1}; \qq))$ 
be a symplectic expansion for the $k$-th copy of $\Sigma_{1,1}$. 
We identify the target with the completed tensor algebra 
$\TT(\qq A_k\oplus \qq B_k)\subset \TT(H_1(\hat S; \qq))$, and 
define a group-like expansion 
$
\hat\theta: \pi_1(\hat S, *) \to \TT(H_1(\hat S; \qq))
$
by $\hat\theta(\chi_k\alpha_k{\chi_k}\inv) := \theta_k(\alpha_k)$ and 
$\hat\theta(\chi_k\beta_k{\chi_k}\inv) := \theta_k(\beta_k)$. 
Then the diagram 
$$
\begin{CD}
\pi_1(S, *) @>{\theta^{\rm std}}>> \TT(H_1(S; \qq))\\
@V{i_*}VV @V{\imath}VV\\
\pi_1(\hat S, *) @>{\hat\theta}>> \TT(H_1(\hat S; \qq))
\end{CD}
$$
commutes, where $i: (S, *) \hookrightarrow (\hat S, *)$ is the inclusion.
We have $\hat\theta(\zeta) = \prod^n_{k=1}\exp(A_kB_k- B_kA_k)
= \imath(\prod^n_{k=1}\exp(x_k))$. 
Here we denote by $u*v$ the Baker-Campbell-Hausdorff series 
of $u$ and $v \in \TT_{\geq 1} = \TT(H_1(S; \qq))_{\geq 1}$ 
$$
u*v := \log((\exp u)(\exp v)) 
= u+ v + \frac12[u,v] + \frac1{12}[u, [u, v]] + \frac1{12}[v, [v, u]] + \cdots,
$$
and consider the element $\Xi := x_1*x_2*\cdots*x_n \in \TT_{\geq 1}$. 
Then we obtain $\log\hat\theta(\zeta) = \imath(\Xi) \in \TT(H_1(\hat S; \qq))$, 
and, from the Massuyeau-Turaev theorem \ref{mt}, 
\begin{equation}
\rho^{\hat\theta}(a, b) = (a- \varepsilon(a))\mt ((-\imath(\Xi))\inv + \omega
s(\imath(\Xi))\omega)\mt (b - \varepsilon(b))
\label{31}
\end{equation}
for any $a, b \in \TT(H_1(\hat S; \qq))$. 
\par
By the injective homomorphism $\imath$, the Massuyeau-Turaev operation 
$\mt$ on $\TT(H_1(\hat S; \qq))$ induces a continuous operation on 
$\TT_{\geq1} = \TT(H_1(\Sigma_{0, n+1}; \qq))_{\geq1}$, 
$
\mt: \TTone\times\TTone \longrightarrow \TTone
$,
given by 
$$
x_{i_1}\cdots x_{i_{l-1}}x_{i_l}\mt x_{j_1}x_{j_2}\cdots x_{j_m}
= -\delta_{i_lj_1} 
x_{i_1}\cdots x_{i_{l-1}} x_{j_1}x_{j_2}\cdots x_{j_m}
$$
for $l, m \geq 1$ and $1 \leq i_1, \dots, i_l, j_1, \dots, j_m \leq n$.
In fact, we have $(A_kB_k- B_kA_k)\mt(A_lB_l- B_lA_l) 
= - \delta_{kl}(A_kB_k- B_kA_k)$ for $1 \leq k, l\leq n$.
The operation $\mt$ on $\TT_{\geq 1}
$ is associative 
with unit $x_0 = -\sum^n_{k=1}x_k$. 
Thus we can take the inverse element $Z\inv$ 
of any $Z \in x_0 + \TT_{\geq 2}$ with respect to 
the operation $\mt$,
$
Z\inv\mt Z = Z\mt Z\inv = x_0.
$
\par
Consider the inverse element $-\Xi\inv$ of $-\Xi = -x_1*x_2*\cdots*x_n$ 
with respect to the operation $\mt$. 
\begin{thm}\label{Y}
$$
-\Xi\inv + x_0s(\Xi)x_0 = x_0 - \sum_{k>l}x_kx_l + \sum^n_{k=1}s(x_k){x_k}^2
= - \sum_{k>l}x_kx_l + \sum^n_{k=1}\frac{{x_k}^2}{e^{-x_k}-1}.
$$
\end{thm}
\begin{proof}
We denote the left-hand side by
$$
Y := - \Xi\inv +x_0s(\Xi)x_0 = \sum^\infty_{m=1} Y_{(m)}, \quad Y_{(m)} \in H^{\otimes m}.
$$
Since $\Xi \equiv -x_0 + \frac12\sum_{k<l} [x_k, x_l] \pmod{\TT_{\geq3}}$, 
we have $Y_{(1)} = x_0$ and 
$$
\aligned
Y_{(2)} & = \frac12\sum_{k<l} [x_k, x_l] - \frac12{x_0}^2\\
& = \frac12\sum_{k<l} (x_kx_l-x_lx_k) - \frac12\sum_{k<l} (x_kx_l+x_lx_k) 
- \frac12\sum^n_{k=1}{x_k}^2\\
& = - \sum_{k>l} x_kx_l - \frac12\sum^n_{k=1}{x_k}^2.
\endaligned
$$
To compute the higher degree term $Y_{(m)}$ for each $m \geq 3$, 
we introduce a topological algebra automorphism $Q$ of $\TT$ 
defined by 
$$
Q(x_k) = -x_{n-k}, \quad 1 \leq k \leq n,
$$
inspired by Kuno's work \cite{Ku2}. See also \cite{Ku1} Example 5.3. 
It is clear to see $Q(\Xi) = -\Xi$ and $Qx_0 = -x_0$. 
Here we have 
\begin{equation*}
Q(u\mt v) = -(Qu)\mt(Qv)
\end{equation*}
for any $u$ and $v \in \TTone$. In fact, we compute 
$(Qx_k)\mt(Qx_l) = (-x_{n-k})\mt(-x_{n-l}) = -\delta_{kl}x_{n-k}
= Q(\delta_{kl}x_k) = -Q(x_k\mt x_l)$ for any $1 \leq k, l \leq n$. 
In particular, for any $Z \in x_0 + \TT_{\geq 2}$, we have 
$x_0 = -Qx_0 = -Q(Z\mt Z\inv) = (QZ)\mt
(QZ\inv)$, and so $Q(Z\inv) = (QZ)\inv$. 
Moreover we have $s(-z) = -1 -s(z)$. Therefore
\begin{equation}
QY = -(Q\Xi)\inv + x_0s(Q\Xi)x_0 
= \Xi\inv - {x_0}^2 - x_0s(\Xi)x_0 = -Y - {x_0}^2.
\label{QY}
\end{equation}
On the other hand, we have 
\begin{equation}
Y\inv = -1 + e^{-\Xi} = -1 + e^{-x_n}\cdots e^{-x_2} e^{-x_1}.
\label{Y-1}
\end{equation}
In fact, $\Xi = \Xi \mt Y\mt Y\inv 
= -\Xi\mt\Xi\inv\mt Y\inv + \Xi s(\Xi)Y\inv
= -Y\inv + \Xi s(\Xi)Y\inv
= \dfrac{\Xi}{e^{-\Xi}-1}Y\inv
$.
Since the algebra $\TT$ has no zero divisor, we obtain (\ref{Y-1}).
\par
Let $W$ (resp.\ $I$) be the closed linear subspace in $\TTone$ 
generated by the set $\{x_{k_1}x_{k_2}\cdots\linebreak x_{k_m}; \,\, 
k_1\geq k_2\geq \dots \geq k_m\}$
(resp.\ $\{x_{k_1}x_{k_2}\cdots x_{k_m}; \,\, 
\sharp\{k_1, k_2, \dots, k_m\} \geq 2\}$).
The subspace $W$ (resp.\ $I$) is a subalgebra (resp.\ a two-sided ideal) 
of $\TTone$ with respect to the multiplication $\mt$. 
Since 
\begin{equation*}
Y = x_0 + \sum^\infty_{m=1}
\overbrace{(x_0 - Y\inv)\mt(x_0 - Y\inv)\mt\cdots \mt (x_0- Y\inv)}^{\text{$m$ times}}.
\end{equation*}
and $x_0 - Y\inv \in W$ from (\ref{Y-1}), we have $Y \in W$. 
It is clear that the direct sum decomposition 
$W = (W\cap I) \oplus \bigoplus^n_{k=1}x_k\qq[[x_k]]$ holds, and so 
$W \cap \Ker(Q+1) \subset \bigoplus^n_{k=1} x_k\qq[[{x_k}]]$, 
while we have $Q(Y - Y_{(2)}) = -(Y - Y_{(2)})$ from (\ref{QY}).
Hence we have $Y - Y_{(2)} \in \bigoplus^n_{k=1}x_k\qq[[x_k]]$.
This implies that it suffices to show the theorem 
modulo the ideal $I$. From (\ref{Y-1}) we have 
$$
\aligned
& Y\inv\mt(x_0+ \sum^n_{k=1}{x_k}^2s(x_k))
= Y\inv\mt (\sum^n_{k=1}x_k\dfrac{x_k}{e^{-x_k} - 1})\\
\equiv &(\sum^n_{k=1}e^{-x_k} -1) \mt
(\sum^n_{k=1}x_k\dfrac{x_k}{e^{-x_k} - 1})
= - \sum^n_{k=1}(e^{-x_k} - 1)\dfrac{x_k}{e^{-x_k} - 1} = x_0.
\endaligned
$$
Hence we have $Y \equiv x_0+ \sum^n_{k=1}{x_k}^2s(x_k) \pmod{I}$, 
as was to be shown.
\end{proof}

As a corollary, we conclude
\begin{equation}
\rho^{\theta^{\rm std}}(a, b) = 
(a- \varepsilon(a))\mt (
- \sum_{k>l}x_kx_l + \sum^n_{k=1}\frac{{x_k}^2}{e^{-x_k}-1}
)\mt (b - \varepsilon(b))
\label{32}
\end{equation}
for any $a, b \in \TT = \TT(H_1(S; \qq))$. 
In particular, by (\ref{kp}), we have
\begin{eqnarray}
\kappa^{\rm std}(x_k, x_l) &=& - ((1\wotimes \iota)\Delta\left(\epsilon_{kl}x_{k}x_{l} - 
\delta_{kl}\frac{{x_{k}}^2}{e^{-x_{k}}-1}\right)  = 
- K_{k,l} \in \TT\wotimes\TT,
\label{33}
\end{eqnarray}
where $\kappa^{\rm std}$ is the tensorial description of $\kappa$ 
with respect to the standard exponential expansion $\theta^{\rm std}$. 
Recall $x_k = \log\theta^{\rm std}(\gamma_k)$. Consequently, 
substituting (\ref{33}) and Theorem \ref{fkk} to the product formula (*), 
we obtain Theorem \ref{mustd}.
This completes the proof.
\qed


\vskip 10mm
\noindent
Department of Mathematical Sciences, \\
University of Tokyo \\
3-8-1 Komaba, Meguro-ku, Tokyo, \\
153-8914, JAPAN. \\
kawazumi@ms.u-tokyo.ac.jp\\
\\
\end{document}